\documentclass[12pt,a4paper]{amsart}
\usepackage{cmap} 
\usepackage[T2A,T1]{fontenc} 

\usepackage{amssymb}
\usepackage{amsthm}
\usepackage{amsmath}
\usepackage{mathrsfs}
\usepackage[pdftex,unicode]{hyperref}
\usepackage[noadjust]{cite}

\usepackage{stackengine}
\stackMath
\def\cl#1{\stackengine{0pt}{}{\overline{\phantom{#1}}}{O}{l}{F}{T}{L}#1}

\newcommand\abst[1]{%
  \setbox0=\hbox{$\skew3\cl{#1}$}%
  \ht0=\dimexpr\ht0-.15ex\relax 
  \skew3\cl{\copy0}}

\def\Int{\operatorname{Int}}
\def\supp{\operatorname{supp}}

\newtheorem{theorem}{Theorem}[section]
\newtheorem{lemma}[theorem]{Lemma}
\newtheorem{proposition}[theorem]{Proposition}
\newtheorem{corollary}[theorem]{Corollary}

\theoremstyle{definition} 
\newtheorem{definition}[theorem]{Definition}
\newtheorem{question}{Question}

\theoremstyle{remark} 
\newtheorem{remark}{Remark}

\begin{document}

{
\renewcommand*{\thefootnote}{$\star$}

\title[$\varkappa$-Fr\'{e}chet--Urysohn and Baire subgroups]{$\varkappa$-Fr\'{e}chet--Urysohn and Baire subgroups\\ of free topological groups}
\footnotetext[0]{The work of O.~Sipacheva was financially supported by the Russian Science Foundation, grant~22-11-00075-P.} 

}

\author[Leiderman, Reznichenko, and Sipacheva]{Arkady Leiderman, Evgenii Reznichenko, and Ol'ga Sipacheva}

\address{Department of Mathematics, Ben-Gurion University of the Negev, 
Beer Sheva, P.O.B. 653, Israel}
\email{arkady@math.bgu.ac.il}  

\address{Department of General Topology and Geometry, Faculty of Mechanics and Mathematics, M.~V.~Lomonosov Moscow State University, Leninskie Gory 1, Moscow, 119991, Russia}
\email{erezn@inbox.ru}

\address{Department of General Topology and Geometry, Faculty of Mechanics and Mathematics, M.~V.~Lomonosov Moscow State University, Leninskie Gory 1, Moscow, 119991, Russia}
\email{ovsipa@gmail.com}

\begin{abstract}
A topological space is $\varkappa$-Fr\'echet--Urysohn if every \tolerance1000 point in the closure of an open set $U$ is the 
limit of a sequence of points of~$U$. We prove that every $\varkappa$-Fr\'echet--Urysohn subgroup and every Baire 
subgroup of the free topological group on a Tychonoff space $X$ in any full variety of topological groups (in 
particular, of the free topological group $F(X)$, the free Abelian topological group $A(X)$, and the free Boolean 
topological group $B(X)$) is discrete. In particular, each of these groups is $\varkappa$-Fr\'echet--Urysohn (or 
Baire) only if $X$ is discrete, and all their metrizable and even Fr\'echet--Urysohn subgroups are discrete. 
Along the way we characterize strongly $\varkappa$-Fr\'echet--Urysohn spaces in the spirit of Michael's theorem 
and study countable families of closed sets absorbing convergent sequences in $\varkappa$-Fr\'echet--Urysohn 
spaces and groups.
\end{abstract}



\maketitle

\section*{Introduction}

Recall that a topological space $X$ is {\em Fr\'{e}chet--Urysohn} if, for any $A\subseteq X$ and $x\in 
\cl{A}$, there exists a sequence of elements of $A$ that converges to~$x$, and {\em 
$\varkappa$-Fr\'{e}chet--Urysohn} if this holds for all open sets $A$. The latter notion was introduced by 
Arhangel'skii in 1999 as a natural extension of the notion of a space of countable $\varkappa$-pseudocharacter, 
and it has become increasingly important in topology, topological algebra, and functional analysis 
\cite{LiL, Sak2, Sakai-3, GOR2026kFUgr, GR2026New_classes, G2019-Baire_functions, GOR-Baire, Osipov2025kfu, 
KrupskiKucharskiMarciszewski2025}. Every metrizable space is Fr\'echet--Urysohn, and every Fr\'echet--Urysohn 
space is $\varkappa$-Fr\'echet--Urysohn. Any locally compact (and, more generally, feathered) topological group 
is $\varkappa$-Fr\'{e}chet--Urysohn \cite[Theorem 6.1]{GR2026New_classes}, and a topological group containing a 
dense $\varkappa$-Fr\'{e}chet--Urysohn subspace is $\varkappa$-Fr\'{e}chet--Urysohn 
\cite[Corollary~2.2]{G2019-Baire_functions}. In particular, if a topological group contains a dense 
Fr\'echet--Urysohn subspace, then it is $\varkappa$-Fr\'echet--Urysohn.

Morris and Thompson \cite[Corollary~1]{MT1986} proved that every Fr\'echet--Urysohn subgroup of the free 
topological group on a Tychonoff space is discrete. We develop this result in two directions: first, we 
consider $\varkappa$-Fr\'echet--Urysohn subgroups, and second, we consider free topological groups on Tychonoff 
spaces in arbitrary full varieties of topological groups. Furthermore, we study Baire subgroups of free 
topological groups. The main result of this paper is that all $\varkappa$-Fr\'echet--Urysohn and all Baire 
subgroups of the free topological group $F^{\mathscr{V}}(X)$ on a Tychonoff space $X$ in a full variety 
$\mathscr{V}$ of topological groups are discrete (Theorems~\ref{t2} and~\ref{t3}); consequently, 
$F^{\mathscr{V}}(X)$ is $\varkappa$-Fr\'echet--Urysohn or Baire if and only if $X$ is discrete 
(Corollary~\ref{cor:FVX-kfu}). In particular, if a subgroup of the free topological group $F(X)$ contains a dense 
Fr\'echet--Urysohn subspace, then it is discrete.

To obtain the main results, we develop the theory of $\varkappa$-Fr\'echet--Ury\-sohn spaces and 
$\varkappa$-Fr\'echet--Urysohn topological groups. By analogy with strongly Fr\'echet--Urysohn spaces, we 
introduce strongly $\varkappa$-Fr\'echet--Urysohn spaces and prove an analogue of Nyikos' theorem: every 
$\varkappa$-Fr\'echet--Urysohn topological group is strongly $\varkappa$-Fr\'echet--Ury\-sohn (Theorem~\ref{t1}). 
Strongly $\varkappa$-Fr\'echet--Urysohn spaces are quite similar to strongly Fr\'echet--Urysohn spaces; in 
particular, $X$ is strongly $\varkappa$-Fr\'echet--Urysohn if and only if $X\times[0,1]$ is 
$\varkappa$-Fr\'echet--Urysohn (Proposition~\ref{p2}). Homogeneous hemicompact $\varkappa$-Fr\'echet--Urysohn 
spaces are locally compact (Corollary~\ref{c2}), which extends a result of~\cite{GOR2026kFUgr}. 

The paper is organized as follows. The brief preliminaries section is followed by Section~\ref{sec:kfu} devoted 
to $\varkappa$-Fr\'echet--Urysohn and strongly $\varkappa$-Fr\'echet--Urysohn topological spaces and groups. 
In Section~\ref{sec:absorb} we introduce and study countable families of closed sets absorbing convergent 
sequences, which play a key role in the proof of the main results. Section~\ref{sec:free_abstract_groups} is 
purely algebraic; it describes some properties of the naturally defined support of an element of a free group in 
an arbitrary variety of abstract groups. In Section~\ref{sec:free_top_groups}, we prove the main theorems. 
Section~\ref{sec:questions} contains open questions.

Similar methods apply to subgroups and linear subspaces of topological vector spaces; these applications are 
the subject of the companion paper~\cite{LRS-tvs}.

\section{Preliminaries}

All topological spaces under consideration are assumed to be Tychonoff. In what follows, we sometimes identify 
sequences with the indexed sets of their values. We use the standard notation $\omega$ for the set of nonnegative 
integers. Given a subset $A$ of a group $G$, by $\langle A\rangle_G$ we denote the smallest subgroup of $G$ 
containing~$A$; we write $\langle A\rangle$ instead of $\langle A\rangle_G$ when it is clear which group $G$ is 
meant. For subsets $A$ and $B$ of a group and a positive integer $m$, by $AB$ we denote the set $\{ab:a\in A,\ 
b\in B\}$ and by $A^m$, the set of all products of $m$ elements of~$A$.

We use the notation $\cl A$ for the closure of a subset $A$ of a topological space, $\Int A$ for its 
interior, and $|A|$ for the cardinality of~$A$.

A space $Z$ is the \emph{inductive limit} of an increasing sequence $(K_n)_{n\in\omega}$ of its subspaces if 
$Z=\bigcup_{n\in\omega}K_n$ and a set $F\subset Z$ is closed in $Z$ if and only if $F\cap K_n$ is closed in $K_n$ 
for every $n\in\omega$.

A subset $A$ of a topological group $G$ is \emph{precompact} in $G$ if, for every neighborhood $O$ of the 
identity element of $G$, there is a finite set $M\subset G$ such that $A\subset MO$, or, equivalently, the closure 
of $A$ in the Ra\u\i kov completion $\varrho G$ of $G$ is compact \cite[Theorem 3.7.10]{AT}. Since the 
Ra\u\i kov completion of a subgroup $H$ of $G$ is the closure of $H$ in $\varrho G$, a subset of $H$ is 
precompact in $H$ if and only if it is precompact in~$G$. A topological group is \emph{locally precompact} 
if its identity element has a precompact neighborhood.

We need the following version of the $\Delta$-system lemma for countable families of sets 
(see, e.g., \cite[Problems~1 and~3 on pp.~107, 421,~422]{Komjath-Totik}).

\begin{proposition}
\label{prop:Delta-system} Let $n\in\omega$, and let $\mathscr{A}$ be an infinite family of sets such that $|M|\le 
n$ for all $M\in \mathscr{A}$. Then there exist a set $R$ and an infinite family $\mathscr{B}\subset \mathscr{A}$ 
of sets of the same cardinality such that $A\cap B=R$ for any distinct $A,B\in \mathscr{B}$. Moreover, 
$A\setminus R\neq\varnothing$ for all $A\in\mathscr{B}$. \end{proposition}

\section{$\varkappa$-Fr\'echet--Urysohn spaces}
\label{sec:kfu}

A topological space $X$ is {\em strongly Fr\'{e}chet--Urysohn} if, given any $x\in X$ and any decreasing sequence 
$(A_n)_{n\in \omega}$ of subsets of $X$ for which $x\in \bigcap_{n\in \omega} \cl A_n$, there is a sequence 
$(x_n)_{n \in\omega}$ converging to $x$ such that $x_n \in  A_n$ for all $n \in \omega$~\cite{Siwiec}.

\begin{definition}
We say that a topological space $X$ is {\it strongly $\varkappa$-Fr\'{e}chet--Urysohn} if, given any $x\in X$ and 
any decreasing sequence $(U_n)_{n\in \omega}$ of open subsets of $X$ for which $x\in \bigcap_{n\in \omega} 
\cl U_n$, there exists a sequence $(x_n)_{n \in\omega}$ converging to $x$ such that $x_n \in  U_n$ for all 
$n \in \omega$. \end{definition}

Clearly, any strongly $\varkappa$-Fr\'{e}chet--Urysohn space is $\varkappa$-Fr\'{e}chet--Ury\-sohn.
In \cite{Arhangelskij1972}, strongly Fr\'{e}chet--Urysohn spaces were characterized as 
Fr\'{e}chet--Urysohn spaces with the $(\alpha_4)$ property. Recall that a topological space $X$ has the \emph{$(\alpha_4)$ 
property} if, for any $x\in X$ and any sequence $(\xi_n)_{n\in\omega}$ of sequences converging to $x$, there 
exists a sequence $\xi$ converging to $x$ and such that $|\{n\in\omega: \xi\cap \xi_n\neq\varnothing\}|=\omega$. 
A similar characterization holds for strongly $\varkappa$-Fr\'{e}chet--Urysohn spaces.

\begin{proposition} 
\label{p1}
For a topological space $X$, the following conditions are equivalent:
\begin{enumerate}
\item[{\rm(i)}] 
$X$ is strongly $\varkappa$-Fr\'{e}chet--Urysohn;
\item[{\rm(ii)}]
given any $x\in X$ and any sequence $(U_n)_{n\in\omega}$ of open subsets of $X$ for which 
$x\in \bigcap_{n\in \omega} \cl U_n$, there exist a sequence $(x_k)_{k\in\omega}$ converging to $x$ and 
a strictly increasing sequence $(n_k)_{k\in\omega}$ of nonnegative integers such that 
$x_k\in U_{n_k}$ for all $k\in\omega$.
\end{enumerate}
\end{proposition}

\begin{proof}
(i)\,$\Rightarrow$\,(ii):\enspace
For each $n\in\omega$, let $W_n=\bigcup_{k\ge n} U_k$. The sets $W_n$ are open and decreasing, and 
$x\in\cl W_n$ for all $n$.
It follows from (i) that there is a sequence $(y_n)_{n \in\omega}$ converging to $x$ 
and such that $y_n \in W_n$ for all $n \in \omega$, i.e., $y_n\in U_{m_n}$ for some $m_n\geq n$. 
Since $m_n\to\infty$, we can choose indices $j_0<j_1<\dots$ so that $m_{j_0}<m_{j_1}<\dots$\,. 
It remains to set $x_k=y_{j_k}$ and $n_k=m_{j_k}$.

(ii)\,$\Rightarrow$\,(i):\enspace 
Let $x\in X$, and let $(U_n)_{n\in \omega}$ be a decreasing sequence of open subsets in $X$ for which $x\in 
\bigcap_{n\in \omega} \cl U_n$. By (ii) there exist a sequence $(x_k)_{k\in\omega}$ converging to $x$ and 
a strictly increasing sequence $(n_k)_{k\in\omega}$ such that $x_k\in U_{n_k}$. Since $n_k\geq k$ and 
the sequence $(U_n)$ decreases, we have $x_k\in U_k$ for all $k\in\omega$.
\end{proof}

In \cite{Michael1972} Michael  characterized strongly Fr\'{e}chet--Urysohn spaces as countably bi-sequential 
spaces and proved that a space $X$ is strongly Fr\'echet--Urysohn if and only if $X \times [0, 1]$ 
is 
Fr\'{e}chet--Urysohn. Strongly $\varkappa$-Fr\'{e}chet--Urysohn spaces have a similar characterization. In 
the following proposition, $\omega+1$ denotes the ordinal $\omega+1$ with the linear order topology, i.e., the 
sequence of nonnegative integers converging to $\omega$ together with its limit~$\omega$. 

\begin{proposition} 
\label{p2}
For a topological space $X$, the following conditions are equivalent:
\begin{enumerate}
\item[{\rm(i)}] 
$X$ is strongly $\varkappa$-Fr\'{e}chet--Urysohn;
\item[{\rm(ii)}] 
$X\times (\omega + 1)$ is $\varkappa$-Fr\'{e}chet--Urysohn;
\item[{\rm(iii)}] 
$X\times [0,1]$ is $\varkappa$-Fr\'{e}chet--Urysohn;
\item[{\rm(iv)}] 
$X\times Y$ is $\varkappa$-Fr\'{e}chet--Urysohn for any first-countable space~$Y$;
\item[{\rm(v)}] 
$X\times Y$ is strongly $\varkappa$-Fr\'{e}chet--Urysohn for any first-countable space~$Y$.
\end{enumerate}
\end{proposition}

\begin{proof}
(ii)\,$\Rightarrow$\,(i):\enspace 
In view of Proposition~\ref{p1}, it suffices to show that if $x\in X$ and $(U_n)_{n\in\omega}$ is a sequence of 
open subsets in $X$ for which $x\in \bigcap_{n\in \omega} \cl U_n$, then condition~(ii) of 
Proposition~\ref{p1} holds. We set $U=\bigcup_{n\in\omega}U_n\times \{n\}$. Then $U$ is open and $(x,\omega)\in 
\cl U$. From (ii) it follows that there exists a sequence $\bigl((y_j,m_j)\bigr)_{j\in \omega}\subset U$ 
converging to $(x,\omega)$; thus, $y_j\in U_{m_j}$, $y_j\to x$, and $m_j\to\infty$. Passing to a subsequence, we 
can assume that the $m_j$ strictly increase, as required. 

(iii)\,$\Rightarrow$\,(i):\enspace 
The proof is similar to that of (ii)\,$\Rightarrow$\,(i). Let $x\in X$, and let $(U_n)_{n\in\omega}$ be a sequence of 
open subsets of $X$ for which $x\in\bigcap_{n\in\omega}\cl U_n$. We set 
$U=\bigcup_{n\in\omega}U_n\times\bigl(\frac1{n+2},\frac1{n+1}\bigr)$. Then $U$ is open in $X\times[0,1]$ and 
$(x,0)\in\cl U$. By (iii) there exists a sequence $\bigl((y_j,t_j)\bigr)_{j\in\omega}\subset U$ converging to 
$(x,0)$. Let $m_j$ be such that $(y_j,t_j)\in U_{m_j}\times\bigl(\frac1{m_j+2},\frac1{m_j+1}\bigr)$. Then $y_j\in U_{m_j}$, 
$y_j\to x$, and $m_j\to\infty$, because $t_j\to0$. Passing to a subsequence, we can assume that the $m_j$ 
strictly increase, so that condition~(ii) of Proposition~\ref{p1} holds. 

(i)\,$\Rightarrow$\,(iv):\enspace 
Take $(x,y)\in X\times Y$ and an open subset $U$ of $X\times Y$ for which $(x,y)\in \cl U$.
Let $\{V_n:n\in\omega\}$ be a decreasing countable base of open neighborhoods of $y$, and let 
\begin{align*}
\pi_X &\colon X\times Y \to X,\ (x,y)\mapsto x,
\\
\pi_Y &\colon X\times Y \to Y,\ (x,y)\mapsto y
\end{align*}
be the projections. We set $U_n=\pi_X(U\cap \pi^{-1}_Y(V_n))$ for $n\in\omega$. Then $x\in \cl U_n$ and 
$U_{n+1}\subset U_n$ for $n\in\omega$. From (i) it follows that there exists a sequence $(x_n)_{n \in \omega}$ 
converging to $x$ and such that $x_n \in U_n$ for all $n \in \omega$. For each $n \in \omega$, choose a $y_n \in 
V_n$ such that $(x_n,y_n) \in U$. The sequence $\bigl((x_n,y_n)\bigr)_{n \in \omega}$ converges to $(x,y)$.

The implications (iv)\,$\Rightarrow$\,(ii), (iv)\,$\Rightarrow$\,(iii), and (v)\,$\Rightarrow$\,(iv) are obvious. 

It remains to prove that (iv)\,$\Rightarrow$\,(v). Let $Y$ be a first-countable space. Then $Y\times 
(\omega+1)$ is first-countable and (iv) implies that $X\times Y\times (\omega+1)$ is 
$\varkappa$-Fr\'{e}chet--Urysohn. Since (ii)\,$\Rightarrow$\,(i), it follows that $X\times Y$ is strongly 
$\varkappa$-Fr\'{e}chet--Urysohn. 
\end{proof}

In \cite{Nyikos1981} Nyikos proved that any Fr\'{e}chet--Urysohn group is strongly Fr\'{e}chet. An immediate 
corollary of Proposition~\ref{p2} and Theorem~2.1 of \cite{GOR2026kFUgr} is the following analogue of Nyikos's 
theorem.

\begin{theorem} 
\label{t1}
Any $\varkappa$-Fr\'{e}chet--Urysohn group is strongly $\varkappa$-Fr\'{e}\-chet--Ury\-sohn.
\end{theorem}

\section{Families absorbing convergent sequences}
\label{sec:absorb}

We say that a family $\mathscr F$ of subsets of $X$ \emph{absorbs convergent sequences} in $X$ if, for every 
convergent sequence $(x_n)_{n\in \omega}$ in $X$, there exists an $F\in \mathscr F$ such that $\{x_n: n\in 
\omega\}\subset F$. Note that $\bigcup\mathscr F=X$ for any such family~$\mathscr F$, because every point $x\in 
X$ generates the constant sequence converging to~$x$. 

\begin{remark}
\label{rem:finer}
If $X$ is a topological space with topology $\tau$, $\tau'$ is another topology on $X$ finer than $\tau$, and a 
family $\mathscr F$ of subsets of $X$ absorbs convergent sequences in $(X,\tau)$, then $\mathscr F$ absorbs 
convergent sequences in $(X, \tau')$. Similarly, if $\mathscr F$ absorbs convergent sequences in $X$ and 
$Y\subset X$, then $\{F\cap Y:F\in\mathscr F\}$ absorbs convergent sequences in~$Y$. 
\end{remark}

\begin{lemma}
\label{lemma:ind-limit}
Let a space $Z$ be the inductive limit of an increasing sequence $(K_n)_{n\in\omega}$ of its closed subspaces. 
Then the family $\{K_n:n\in\omega\}$ absorbs convergent sequences in~$Z$.
\end{lemma}

\begin{proof}
Suppose that a sequence $(x_k)_{k\in\omega}$ converges to $x\in Z$ but is contained in no $K_n$. 
Any finite subset of $Z$ is contained in some $K_n$, because the $K_n$ increase and cover $Z$. Hence, for each 
$j\in\omega$ and each $p\in \omega$, there exists a $k>p$ such that $x_k\notin K_j$. This allows us to choose 
indices $k_0<k_1<\dots$ so that $x_{k_j}\notin K_j$ for all $j$. Let $m\in\omega$ be such that $x\in K_m$, 
and let $A=\{x_{k_j}:j\geq m\}$. Then $x\notin A$, because $x_{k_j}\notin K_j\supset K_m$ for $j\geq m$, and 
$x\in \cl A$, because $x_{k_j}\to x$. On the other hand, $A\cap K_n\subset\{x_{k_j}:j<n\}$ is finite and hence 
closed in $K_n$ for each $n$; therefore, $A$ is closed in $Z$. This contradiction proves the lemma.
\end{proof}

\begin{proposition} 
\label{p3}
Suppose that a family $\mathscr F=\{F_n: n\in \omega\}$ of closed subsets of a topological space $X$ absorbs 
convergent sequences in~$X$ and $F_n\subset F_{n+1}$ for all $n\in\omega$.
\begin{enumerate}
\item[{\rm(i)}] 
If $X$ is $\varkappa$-Fr\'{e}chet--Urysohn, then $\bigcup_{n\in\omega}\Int F_n$ is dense in $X$.
\item[{\rm(ii)}] 
If $X$ is strongly $\varkappa$-Fr\'{e}chet--Urysohn, then $\bigcup_{n\in\omega}\Int F_n=X$.
\end{enumerate}
\end{proposition}

\begin{proof}
(i)\enspace 
Suppose that, on the contrary, $Y=X\setminus \cl{\bigcup_{n\in\omega}\Int F_n}\neq\varnothing$. Note that 
$Y$ is open. No point $y\in Y$ is isolated in $Y$: otherwise, $\{y\}$ is open in $X$, and since $y\in F_n$ for 
some $n$, we have $y\in \Int F_n$. Take $y\in Y$. The set $Y\setminus\{y\}$ is open and $y\in 
\cl{Y\setminus\{y\}}$. Since $X$ is $\varkappa$-Fr\'{e}chet--Urysohn, there exists a sequence $(y_n)_{n\in 
\omega}\subset Y\setminus\{y\}$ converging to $y$. For each $n\in\omega$, choose an open set $U_n$ so that 
$y_n\in U_n\subset Y$ and $y\notin \cl U_n$. We set $V_n=U_n\setminus F_n$ and $V=\bigcup_{n\in\omega}V_n$. 
The open set $U_n$ does not meet $\Int F_n$; hence $V_n$ is dense in $U_n$ and $y_n\in \cl V_n$. Since 
$y_n\to y$, it follows that $y\in\cl V$, and since $X$ is $\varkappa$-Fr\'{e}chet--Urysohn, there exists a 
sequence $(x_k)_{k\in \omega}$, $x_k\in V$, converging to~$y$.

Let us show that $\{x_k:k\in\omega\}\not\subset F_m$ for every $m\in\omega$; this will contradict the assumption 
that $\mathscr F$ absorbs convergent sequences. Since $y\notin\cl{U_0\cup\dots\cup U_m}$, only finitely 
many $x_k$ belong to $U_0\cup\dots\cup U_m$. Take $k$ with $x_k\notin U_0\cup\dots\cup U_m$. Then $x_k\in V_j$ 
for some $j>m$, so that $x_k\notin F_j\supset F_m$.

(ii)\enspace 
Assume that, on the contrary, $Y=X\setminus \bigcup_{n\in\omega}\Int F_n\neq\varnothing$. Take $y\in Y$. Let 
$V_n=X\setminus F_n$ for $n\in \omega$. The open sets $V_n$ decrease, and $y\in \cl V_n$, because $y\notin 
\Int F_n$. Since $X$ is strongly $\varkappa$-Fr\'{e}chet--Urysohn, there exists a sequence $(x_n)_{n\in\omega}$ 
converging to $y$ and such that $x_n\in V_n$, i.e., $x_n\notin F_n$, for $n\in \omega$. The sequence $(x_n)_{n\in 
\omega}$ is not contained in any $F_n$, which contradicts the assumption that $\mathscr F$ absorbs convergent 
sequences. 
\end{proof}

Recall that a topological space $X$ is \emph{hemicompact} if $X = \bigcup_{n\in\omega}K_n$, where all $K_n$ are 
compact and, for every compact set $K \subset X$,  there is an $n\in\omega$ for which $K \subset K_n$.

\begin{corollary} 
\label{c1}
Let $X$ be a hemicompact space.
\begin{enumerate}
\item[{\rm(i)}] 
If $X$ is $\varkappa$-Fr\'{e}chet--Urysohn, then $X$ has a dense subset of points of local compactness.
\item[{\rm(ii)}] 
If $X$ is strongly $\varkappa$-Fr\'{e}chet--Urysohn, then $X$ is locally compact. 
\end{enumerate}
\end{corollary}

\begin{proof}
Replacing $K_n$ by $K_0\cup\dots\cup K_n$, we can assume that the $K_n$ increase. The family $\{K_n:n\in\omega\}$ 
absorbs convergent sequences, because any convergent sequence together with its limit is compact. It remains to 
apply Proposition~\ref{p3} and note that every point of $\Int K_n$ has the compact neighborhood $K_n$. 
\end{proof}

\begin{corollary} 
\label{c2}
Any hemicompact homogeneous $\varkappa$-Fr\'{e}chet--Urysohn \tolerance1500 space $X$ is locally compact.
\end{corollary}

\begin{proof}
By Corollary~\ref{c1}(i), $X$ has a point of local compactness, and by homogeneity all points of $X$ are points 
of local compactness. 
\end{proof}

In \cite{GOR2026kFUgr}, this corollary was proved for the case of  topological groups. In this case, the 
following two propositions also hold.

\begin{proposition} 
\label{prop:kfu-group-as}
Suppose that a family $\mathscr F=\{F_n: n\in \omega\}$ of closed subsets of a $\varkappa$-Fr\'{e}chet--Urysohn 
topological group $G$ absorbs convergent sequences in~$G$ and $F_n\subset F_{n+1}$ for all $n\in\omega$. Then 
$F_n$ is a neighborhood of the identity element of $G$ for some $n \in \omega$. 
\end{proposition}

\begin{proof}
According to Theorem~\ref{t1}, $G$ is strongly $\varkappa$-Fr\'{e}chet--Urysohn. \tolerance1000
Therefore, $\bigcup_{n\in\omega}\Int F_n=G$ by Proposition~\ref{p3}(ii); in particular, the identity 
element belongs to some $\Int F_n$.
\end{proof}

\begin{proposition} 
\label{prop:kfu-group-loc-precomp}
Suppose that a family $\mathscr F=\{F_n: n\in \omega\}$ of closed precompact subsets of a topological group $G$ 
absorbs convergent sequences in~$G$. Then any $\varkappa$-Fr\'{e}chet--Urysohn subgroup of $G$ is locally 
precompact. 
\end{proposition}

\begin{proof}
Let $H$ be a $\varkappa$-Fr\'{e}chet--Urysohn subgroup of $G$, and let $P_n = (F_0\cup\dots\cup F_n) \cap H$ for 
$n \in \omega$. Then the increasing family $\mathscr{P} = \{P_n : n \in \omega\}$ consists of closed subsets of 
$H$ precompact in $G$ and hence in $H$, and $\mathscr{P}$ absorbs convergent sequences in $H$ (see 
Remark~\ref{rem:finer}). By Proposition~\ref{prop:kfu-group-as} one of the precompact sets $P_n$ is a 
neighborhood of the identity element of $H$. Therefore, $H$ is locally precompact.
\end{proof}

\section{Free groups in varieties of abstract groups}
\label{sec:free_abstract_groups}

We refer to groups without topology as \emph{abstract groups}. In this section, we consider only abstract groups and 
simply call them groups. By $e$ we always denote the identity element of a group under consideration.

A class $\mathscr{V}$ of groups is called a \emph{variety} if it satisfies the following conditions:
\begin{enumerate}
\item[$(\mathscr{V}_0)$]
if $G\in \mathscr{V}$ and a group $H$ is isomorphic to $G$, then $H\in \mathscr{V}$;
\item[$(\mathscr{V}_1)$]
any product of groups in $\mathscr{V}$ belongs to $\mathscr{V}$;
\item[$(\mathscr{V}_2)$]
any subgroup of a group in $\mathscr{V}$ belongs to $\mathscr{V}$;
\item[$(\mathscr{V}_3)$]
any quotient of  a group in $\mathscr{V}$ belongs to $\mathscr{V}$.
\end{enumerate}
In what follows, we assume all varieties under considerations to be \emph{nontrivial}, i.e., contain 
some nontrivial groups.

By Birkhoff's theorem \cite{Birkhoff}, a class $\mathcal K$ of groups is a variety if and only if 
there exists a set of identities such that $\mathcal K$ consists of precisely those groups in which all 
identities in this set hold. Examples of varieties are the classes of all groups, of Abelian groups, 
and of Boolean groups.

As is customary in group theory, by a \emph{word} in an arbitrary subset $X$ of a group $G$ we mean any written 
group product of elements of $X$ and their inverses; the number of factors is called the \emph{length} of the 
word. Given an $x\in \langle X\rangle_G$, we define the \emph{length of $x$ relative to $X$} as the minimum length 
of a word in $X$ representing $x$, that is, 
\begin{align*}
\min\{n\in\omega: {}&\text{there exist $x_1,\dots,x_n\in X$ and $\varepsilon_1,\dots,\varepsilon_n\in\{-1,1\}$}
\\
&\text{for which $x=x_1^{\varepsilon_1}\dots x_n^{\varepsilon_n}$}\}.
\end{align*}
Clearly, for any $x,y\in \langle X\rangle_G$, the length of the product $xy$ does not exceed the sum of the lengths 
of $x$ and~$y$.

The \emph{free group} $F^{\mathscr{V}}(X)$ on a nonempty set $X$ in a variety $\mathscr{V}$ is defined by the 
following conditions: 
\begin{enumerate}
\item[$(\mathscr{F}_1)$]
$F^{\mathscr{V}}(X)\in \mathscr{V}$;
\item[$(\mathscr{F}_2)$]
$F^{\mathscr{V}}(X)$ is generated by $X$; 
\item[$(\mathscr{F}_3)$]
any map $f\colon X\to G$, where $G\in \mathscr{V}$, extends to a homomorphism $\hat f\colon F^{\mathscr{V}}(X)\to G$ 
(which is unique by $(\mathscr{F}_2)$). 
\end{enumerate}
Since $\mathscr{V}$ is nontrivial, the elements of $X$ are pairwise distinct and different from $e$ in 
$F^{\mathscr{V}}(X)$: given a nontrivial $G\in\mathscr{V}$, a $g\in G\setminus\{e\}$, and $x\in X$, the map 
sending $x$ to $g$ and all other points of $X$ to $e$ extends to a homomorphism separating $x$ from $e$ and from 
all $y\in X\setminus\{x\}$. Note that conditions $(\mathscr{F}_1)$ and $(\mathscr{F}_3)$ imply that any map 
$f\colon X\to F^{\mathscr{V}}(X)$ extends to a homomorphism $\hat f\colon F^{\mathscr{V}}(X)\to 
F^{\mathscr{V}}(X)$.

The free  group $F^{\mathscr{V}}(X)$ on a set $X$ in a variety $\mathscr{V}$ is the union $\bigcup_{n\in 
\omega}F_n^{\mathscr{V}}(X)$, where each $F_n^{\mathscr{V}}(X)$ is the subset of $F^{\mathscr{V}}(X)$ consisting 
of all elements represented by words of length at most $n$ in~$X$. For $A\subset X$, we identify $\langle 
A\rangle_{F^\mathscr{V}(X)}$ with $F^\mathscr{V}(A)$ (clearly, $\langle A\rangle_{F^\mathscr{V}(X)}$ satisfies 
conditions $(\mathscr{F}_1)$--$(\mathscr{F}_3)$ with $X$ replaced by~$A$).

\begin{definition}
Given a variety $\mathscr{V}$, a set $X$, and $g\in F^\mathscr{V}(X)$, we put  
$$
\supp g=\bigcap \{A\subset X: g\in \langle A\rangle\}. 
$$
We refer to $\supp g$ as the \emph{support} of $g$ in~$F^\mathscr{V}(X)$.
\end{definition}

\begin{proposition}
\label{prop:property_afg}
For any variety $\mathscr{V}$, set $X$, and $g\in F^\mathscr{V}(X)$, the following assertions hold: 
\begin{enumerate}
\item[\rm (i)]
$g\in \langle\supp g\rangle$;
\item[\rm (ii)]
if $A\subset X$ and $g\in\langle  A\rangle$, then $\supp g\subset A$;
\item[\rm (iii)]
if $n$ is the length of $g$ in $F^\mathscr{V}(X)$ and $g=x_1^{\varepsilon_1}x_2^{\varepsilon_2}\dots x_n^{\varepsilon_n}$ 
for some $x_1,\dots,x_n\in X$ and $\varepsilon_1,\dots,\varepsilon_n\in \{-1,1\}$, then $\supp g=\{x_1,\dots,x_n\}$.
\end{enumerate}
\end{proposition}

To prove the proposition, we need the following lemma.

\begin{lemma}
\label{lemma:afg-intersection}
If ${\mathscr{V}}$ is a variety of abstract groups, $X$ a nonempty set, and $F^{\mathscr{V}}(X)$ 
the free group on $X$ in $\mathscr{V}$, then 
$\langle A\rangle \cap \langle B\rangle =\langle A\cap B\rangle$ for any $A,B\subset X$.
\end{lemma}
\begin{proof}
It suffices to check that $\langle A\rangle \cap \langle B\rangle \subset \langle A\cap B\rangle$. Take 
$g\in \langle A\rangle \cap \langle B\rangle$; we have $g = 
a_1^{\varepsilon_1}\dots a_n^{\varepsilon_n}$, where $n\in \omega$, $a_i\in A$, 
$\varepsilon_i\in \{-1,1\}$ for $i\le n$.
Let $f\colon X\to F^{\mathscr{V}}(X)$ be the map defined by setting $f(x)=x$ for $x\in B$ and $f(x)=e$ 
for $x\notin B$, and let $\hat f\colon F^\mathscr{V}(X)\to F^\mathscr{V}(X)$ be the extension of $f$ to a homomorphism. 
Since $g\in \langle B\rangle$, we have $\hat f(g)=g$. 
Let us number the indices $k\le n$ for which $a_k\in B$ in increasing order: 
$$
\{k\in\{1,\dots,n\}: a_k\in B\}=\{k_1,\dots,k_m\}.
$$
Then
\[
g=\hat f(g)=f(a_1)^{\varepsilon_1}f(a_2)^{\varepsilon_2}\dots f(a_n)^{\varepsilon_n}=a_{k_1}^{\varepsilon_{k_1}}\dots a_{k_m}^{\varepsilon_{k_m}},
\]
where all $a_{k_i}$ belong to $A\cap B$. Therefore, $g\in \langle A\cap B\rangle$.
\end{proof}

\begin{proof}[Proof of Proposition \ref{prop:property_afg}]
The case where $g$ is the identity element of $F^\mathscr{V}(X)$ is trivial. In what follows, we assume that 
$g\ne e$. 

Let $n$ be the length of $g$ in $F^\mathscr{V}(X)$. Then $g=x_1^{\varepsilon_1}x_2^{\varepsilon_2}\dots x_n^{\varepsilon_n}$ 
for some $x_1,\dots,x_n\in X$ and $\varepsilon_1,\dots,\varepsilon_n\in \{-1,1\}$.
We set $\mathscr{A}=\{A\subset X: g\in \langle A\rangle\}$ and $S=\{x_1,\dots,x_n\}$. Then $S\in \mathscr{A}$ and $|S|\leq n$.

(i)\enspace
According to Lemma~\ref{lemma:afg-intersection}, the family $\mathscr{A}$ is closed under finite intersections. 
Let $\mathscr{A}'=\{A\cap S: A\in \mathscr{A}\}$. 
As mentioned, $S\in \mathscr{A}$; consequently, $\mathscr{A}'\subset \mathscr{A}$ and $\supp g=\bigcap \mathscr{A}'$. 
Since the family $\mathscr{A}'$ is finite (because so is $S$) 
and $\mathscr{A}$ is closed under finite intersections, we have $\supp g\in \mathscr{A}$, that is, $g\in \langle \supp g\rangle$.

Assertion (ii) follows from the definition of $\supp g$.

(iii)\enspace
We must show that $\supp g=S$. Clearly, $\supp g\subset S$. 
As in the proof of Lemma~\ref{lemma:afg-intersection}, we 
number the indices $k\le n$ for which $x_k\in \supp g$ in increasing order: 
$$
\{k\in \{1,\dots,n\}: x_k\in \supp g\}=\{k_1,\dots,k_m\}.
$$
Note that $\supp g\ne S$ if and only if $m<n$.  
Let $f\colon X\to F^\mathscr{V}(X)$ be the map defined by setting $f(x)=x$ for $x\in\supp g$ and $f(x)=e$ 
for $x\notin\supp g$, and let $\hat f\colon F^\mathscr{V}(X)\to F^\mathscr{V}(X)$ be the extension of $f$ to a homomorphism. 
Since $g\in \langle \supp g\rangle$, we have $\hat f(g)=g$. Applying $\hat f$ to 
$g=x_1^{\varepsilon_1}\dots x_n^{\varepsilon_n}$, we obtain the product 
$$
f(x_1)^{\varepsilon_1}f(x_2)^{\varepsilon_2}\dots f(x_n)^{\varepsilon_n} = x_{k_1}^{\varepsilon_{k_1}}\dots x_{k_m}^{\varepsilon_{k_m}}.
$$ 
If $m<n$, then the length of $g$ is strictly less than $n$, which contradicts our assumption. Therefore, 
$m=n$ and $\supp g= S$.
\end{proof}

\begin{remark}
\label{rem:supp}
The support of an element $g$ in a free, free Abelian, or free Boolean group is simply the set of letters of the 
reduced word representing~$g$. The definition of $\supp g$ for an arbitrary variety of groups and its properties 
established in Lemma~\ref{lemma:afg-intersection} and Proposition~\ref{prop:property_afg} are very natural and 
may be folklore. The notion of support is closely related to that of an essential variable for an element of 
a universal algebra (see, e.g.,~\cite{Salomaa}). 
\end{remark}

We remind the reader that, given sets $X$ and $A\subset X$, we identify $\langle A\rangle$ with $F^{\mathscr{V}}(A)$.

\begin{proposition}
\label{prop:afgk}
Let $\mathscr{V}$ be a variety and $X$ be a set. Then, for any $k\in \omega$ and $A\subset X$,
\begin{enumerate}
\item[\rm (i)]
$\langle A\rangle \cap F_k^{\mathscr{V}}(X)=F_k^{\mathscr{V}}(A)$;
\item[\rm (ii)]
if $A$ is finite, then so is $\langle A\rangle \cap F_k^{\mathscr{V}}(X)$.
\end{enumerate}
\end{proposition}

\begin{proof}
(i)\enspace
Clearly, $F_k^{\mathscr{V}}(A)\subset \langle A\rangle \cap F_k^{\mathscr{V}}(X)$. 
Let us prove the reverse inclusion. Let $g\in \langle A\rangle \cap F_k^{\mathscr{V}}(X)$, and let 
$n$ be the length of $g$ in $F^{\mathscr{V}}(X)$. Then $n\leq k$.
It follows from Proposition~\ref{prop:property_afg}(ii) that $\supp g\subset A$.
Let $g=x_1^{\varepsilon_1}\dots x_n^{\varepsilon_n}$ be a representation of $g$ by a word of length $n$. By 
Proposition~\ref{prop:property_afg}(iii), $\{x_1,\dots,x_n\}=\supp g\subset A$; thus, this is a word in $A$, 
and the length of $g$ in $F^{\mathscr{V}}(A)$ does not exceed~$n$.
Therefore, $g\in F_k^{\mathscr{V}}(A)$.

(ii)\enspace 
The set $F_k^{\mathscr{V}}(A)$ is obviously finite for finite~$A$.
\end{proof}

\begin{proposition}
\label{prop:afgn-finete}
Let $\mathscr{V}$ be a variety and $X$ be a  set.
If $M\subset F^{\mathscr{V}}(X)$ is infinite, then $M^{n+1}\not\subset F_n^{\mathscr{V}}(X)$ for each $n\in\omega$.
\end{proposition}

\begin{proof}
For $g\in F^{\mathscr{V}}(X)$, let $l(g)$ denote the length of $g$ in $F^{\mathscr{V}}(X)$. By 
Proposition~\ref{prop:property_afg}(iii), $|\supp g|\le l(g)$; by Proposition~\ref{prop:property_afg}(i) 
and~(ii), $\supp(gh)\subset\supp g\cup\supp h$ and $\supp g^{-1}=\supp g$ for $g,h\in F^{\mathscr{V}}(X)$.
Consider two cases.

First, suppose that the set $\{l(g): g\in M\}$ is unbounded. Let $g\in M$. Choose an $h\in M$ so that $l(h)> n 
l(g) + n$. Since $h=h g^n g^{-n}$ and $n l(g) + n<l(h)\leq l(h g^n) + l(g^{-n})\leq l(h g^n) + n l(g)$, we have 
$l(h g^n)>n$, that is, $h g^n\notin F_n^{\mathscr{V}}(X)$. It remains to note that $h g^n\in M^{n+1}$.

Now suppose that the set $\{l(g): g\in M\}$ is bounded, that is, $M\subset F_k^{\mathscr{V}}(X)$ for some $k\in 
\omega$. Let $U=\bigcup\{\supp g: g\in M\}$. Proposition \ref{prop:property_afg}(i) implies the inclusion 
$M\subset \langle U\rangle$. Since $M$ is infinite and $M\subset \langle U\rangle\cap F_k^{\mathscr{V}}(X)$, it 
follows that $\langle U\rangle\cap F_k^{\mathscr{V}}(X)$ is infinite. By Proposition~\ref{prop:afgk}(ii) the set 
$U$ is infinite. Therefore, so is the family $\{\supp g: g\in M\}$. Since $|\supp g|\le l(g)\le k$ for every 
$g\in M$, Proposition~\ref{prop:Delta-system} implies the existence of a finite set $D\subset X$ and elements 
$g_m\in M$, $m\in\omega$, with pairwise distinct supports such that $\supp g_i\cap \supp g_j=D$ for any distinct 
$i,j\in\omega$ and $\supp g_i\setminus D\ne\varnothing$ for all $i$. Let \[ g=g_0 g_1\dots g_n. \] Then $g\in 
M^{n+1}$.

Let us show that $g\notin F_n^{\mathscr{V}}(X)$. In view of Proposition~\ref{prop:property_afg}(iii), it suffices 
to prove that $|\supp g|>n$. For each $i=0,1,\dots,n$, take an $x_i\in \supp g_i\setminus D$. We claim that $x_i 
\in \supp g$ for every $i=0,1,\dots,n$. Indeed, note that $x_i\notin \supp g_j$ for $j\neq i$. We have $g=p g_i 
q$, where $p=g_0 g_1\dots g_{i-1}$ and $q=g_{i+1}\dots g_{n}$. Since $\supp p\subset \bigcup_{j=0}^{i-1}\supp 
g_j$ and $\supp q\subset \bigcup_{j=i+1}^{n}\supp g_j$, it follows that $x_i\notin \supp p$ and $x_i\notin \supp 
q$. We have 
$$
x_i\in\supp g_i= \supp (p^{-1} g q^{-1}) \subset \supp p \cup \supp g \cup \supp q
$$
and $x_i\notin \supp p\cup \supp q$; therefore, $x_i\in\supp g$.

Since $x_i\neq x_j$ for different $i$ and $j$, it follows that $|\supp g|>n$.
\end{proof}

\section{Baire and $\varkappa$-Fr\'echet--Urysohn subgroups\\ of free topological groups}
\label{sec:free_top_groups}

In this section, we turn from abstract groups to topological ones.

Given a topological group $G$, by $\abst G$ we denote the abstract group underlying $G$, that is, the group $G$ 
without topology. Let $\abst{\mathscr{V}}$ be any nontrivial variety of abstract groups, and let $\mathscr{V}$ be 
the variety of all topological groups $G$ with $\abst G\in \abst{\mathscr{V}}$; such varieties of topological 
groups are said to be \emph{full}~\cite{MorrisII}. The \emph{free topological group} on a topological space $X$ 
in $\mathscr{V}$ is a topological group $F^\mathscr{V}(X)$ together with a continuous map $\sigma\colon X\to 
F^\mathscr{V}(X)$ such that 
\begin{enumerate}
\item[$(\mathscr{F\!T}_1)$]
$F^{\mathscr{V}}(X)\in \mathscr{V}$;
\item[$(\mathscr{F\!T}_2)$]
$F^{\mathscr{V}}(X)$ is algebraically generated by $\sigma(X)$; 
\item[$(\mathscr{F\!T}_3)$]
for every continuous map $f\colon X\to G$, where $G\in \mathscr{V}$, there exists a continuous homomorphism $\hat 
f\colon F^{\mathscr{V}}(X)\to G$ for which $f=\hat f\circ\sigma$ (it is unique by $(\mathscr{F\!T}_2)$). 
\end{enumerate}
It is well known that, for any topological space $X$, the free topological group $F^\mathscr{V}(X)$ in 
$\mathscr{V}$ exists and is unique up to topological isomorphism (see, e.g., \cite{Taylor}); this is proved by a 
standard argument based on the diagonal of a set of continuous maps from $X$ to topological groups in 
$\mathscr{V}$. Moreover, $F^\mathscr{V}(X)$ is isomorphic to the abstract free group $F^{\hbox{\scriptsize$\abst{\mathscr{V}}$}}(X)$ 
\cite{Taylor} and if $X$ is Tychonoff, then $F^\mathscr{V}(X)$ is Hausdorff and contains $X$ as a 
subspace~\cite[Theorems~2.1 and~2.2]{MorrisII} (more precisely, $\sigma$ is a homeomorphic embedding, and we 
identify $X$ with $\sigma(X)$). The group $F^\mathscr{V}(X)$ can be represented as the union $\bigcup_{n\in 
\omega}F_n^\mathscr{V}(X)$, where each $F_n^\mathscr{V}(X)$ is the subspace of $F^\mathscr{V}(X)$ consisting of 
all words of length at most $n$ in~$X$. 

\begin{theorem}
\label{new_old_theorem}
For any full variety $\mathscr V$ of topological groups and any Tychonoff space $X$, all sets 
$F_n^\mathscr{V}(X)$, $n\in \omega$, are closed in $F^\mathscr{V}(X)$ and the family $\{F_n^\mathscr{V}(X): n\in 
\omega\}$ absorbs convergent sequences in $F^\mathscr{V}(X)$. 
\end{theorem}

\begin{proof}
Let $K$ be any compact space containing $X$ (e.g., $K=\beta X$). The embedding of $X$ into $K$ extends to a 
continuous homomorphism $i\colon F^\mathscr{V}(X)\to F^\mathscr{V}(K)$; its image is the subgroup $\langle  
X\rangle$ of $F^\mathscr{V}(K)$, which is algebraically the free group $F^{\hbox{\scriptsize$\abst{\mathscr{V}}$}}(X)$ (see 
Section~\ref{sec:free_abstract_groups}), and hence $i$ is injective. We identify $F^\mathscr{V}(X)$ with $\langle 
X\rangle\subset F^\mathscr{V}(K)$ as an abstract group; the topology of $F^\mathscr{V}(X)$ is finer than that 
inherited from $F^\mathscr{V}(K)$, and $F_n^\mathscr{V}(X)=F_n^\mathscr{V}(K)\cap F^\mathscr{V}(X)$ by 
Proposition~\ref{prop:afgk}(i).

Mal'tsev proved that $F^\mathscr{V}(K)$ has the inductive limit topology with respect to the decomposition 
$F^\mathscr{V}(K)= \bigcup_{n\in \omega}F_n^\mathscr{V}(K)$ \cite[Lemma in Section~5]{Maltsev1957}. 
Each set $F_n^\mathscr{V}(K)$ is compact, being the image of the compact space 
$\bigl(K\oplus \{e\}\oplus K^{-1}\bigr)^n$ under the continuous natural multiplication map 
$(x_1,\dots, x_n)\mapsto x_1\dots x_n$. Therefore, each $F_n^\mathscr{V}(K)$ is closed in $F^\mathscr{V}(K)$ and 
hence each $F_n^\mathscr{V}(X)$ is closed in $F^\mathscr{V}(X)$.

By Lemma~\ref{lemma:ind-limit} the family $\{F_n^\mathscr{V}(K): n\in\omega\}$ absorbs convergent sequences in 
$F^\mathscr{V}(K)$. By Remark~\ref{rem:finer} the family $\{F_n^\mathscr{V}(K)\cap F^\mathscr{V}(X): 
n\in\omega\}=\{F_n^\mathscr{V}(X): n\in \omega\}$ absorbs convergent sequences in $F^\mathscr{V}(X)$. 
\end{proof}

\begin{proposition}
\label{new_old_proposition}
For any nondiscrete subgroup $G$ of $F^\mathscr{V}(X)$ and any $n\in \omega$, the intersection $F_n^{\mathscr 
V}(X)\cap G$ is closed and nowhere dense in $G$. 
\end{proposition} 

\begin{proof}
Let $e$ be the identity element of $F^\mathscr{V}(X)$. By Theorem~\ref{new_old_theorem} 
each $F_n^\mathscr{V}(X)$ is closed in 
$F^\mathscr{V}(X)$; therefore, each $F_n^\mathscr{V}(X)\cap G$ is closed in $G$. 
Suppose that there exists a $k$ for which $F_k^\mathscr{V}(X)\cap G$ has nonempty 
interior in $G$, i.e., contains a neighborhood $U$ of some word $g\in F_k^\mathscr{V}(X)\cap G$. 
Then, for some neighborhood $V$ of $e$ in $G$, we have $gV\subset U$. Clearly, 
$V\subset g^{-1}U\subset F_{2k}^\mathscr{V}(X)$. 
Let $W$ be a neighborhood  of $e$ in $G$ for which $W^{2k+1}\subset V$.
Since $G$ is a nondiscrete Hausdorff group, the set $W$ is infinite.
Proposition~\ref{prop:afgn-finete} implies $W^{2k+1}\not\subset F_{2k}^\mathscr{V}(X)$. This  contradiction 
proves Proposition~\ref{new_old_proposition}.
\end{proof}

An immediate corollary of this proposition is the following theorem.

\begin{theorem}
\label{t2}
For any full variety $\mathscr V$ of topological groups and any Tychonoff space $X$, 
every Baire subgroup of $F^\mathscr{V}(X)$ is discrete.
\end{theorem}

Since the family $\{F_n^\mathscr{V}(X)\cap G:n\in\omega\}$ is increasing and absorbs convergent sequences in any 
subgroup $G$ of $F^\mathscr{V}(X)$, Propositions~\ref{prop:kfu-group-as} and~\ref{new_old_proposition}, together 
with Theorem~\ref{new_old_theorem}, imply the following result.

\begin{theorem}
\label{t3}
For any full variety $\mathscr V$ of topological groups and any Tychonoff space $X$, 
every $\varkappa$-Fr\'echet--Urysohn subgroup of $F^\mathscr{V}(X)$ is discrete.
\end{theorem}

Recall that, given a topological space $X$, the standard notations $F(X)$, $A(X)$, and $B(X)$ are used for 
its free topological groups 
in the varieties of all topological groups, all Abelian topological groups, and all 
Boolean (that is, of exponent 2) topological groups, respectively.

\begin{corollary}
\label{cor:FAB}
 All Baire and all $\varkappa$-Fr\'echet--Urysohn subgroups of \tolerance1000 $F(X)$, $A(X)$, and $B(X)$ are discrete 
for any Tychonoff space~$X$.
\end{corollary}

\begin{corollary}
\label{cor:FVX-kfu}
For a Tychonoff space $X$ and a full variety $\mathscr{V}$, the following conditions are equivalent: 
\begin{enumerate}
\item[\rm(i)] $X$ is discrete;
\item[\rm(ii)] $F^{\mathscr{V}}(X)$ is discrete;
\item[\rm(iii)] $F^{\mathscr{V}}(X)$ is $\varkappa$-Fr\'echet--Urysohn;
\item[\rm(iv)] $F^{\mathscr{V}}(X)$ is Baire.
\end{enumerate}
\end{corollary}

\begin{proof}
If $X$ is discrete, then the abstract group $F^{\hbox{\scriptsize$\abst{\mathscr{V}}$}}(X)$ with the discrete topology belongs to 
$\mathscr{V}$, and the identity map of $X$ extends to a continuous isomorphism of $F^\mathscr{V}(X)$ onto it; 
hence $F^\mathscr{V}(X)$ is discrete. Clearly, (ii) implies (iii) and (iv). If (iii) or (iv) holds, then 
$F^\mathscr{V}(X)$ is discrete by Theorem~\ref{t3} or~\ref{t2}, and so is its subspace~$X$. 
\end{proof}

\section{Open questions}
\label{sec:questions}

As far as we know, the following questions remain open.

\begin{question}
\label{q1}
Is it true that every $\varkappa$-Fr\'{e}chet--Urysohn topological group contains a dense Fr\'{e}chet--Urysohn subspace?
\end{question}

\begin{question}
\label{q2}
Is it true that every compact topological group contains a dense Fr\'{e}chet--Urysohn subspace?
\end{question}

\begin{question}[{\cite[Problem~4.1.1]{AT}}]
\label{q3}
Is it true that every compact topological group contains a dense 
countably tight subspace?
\end{question}

\begin{question}[{\cite[Problem~4.1.2]{AT}}]
\label{q4}
Is it true that every compact topological group contains a dense 
sequential subspace?
\end{question}

Since compact groups are $\varkappa$-Fr\'echet--Urysohn, Question~\ref{q2} is a special case of 
Question~\ref{q1}. A positive answer to Question~\ref{q2} would imply positive answers to 
Questions~\ref{q4} and~\ref{q3}, because Fr\'echet--Urysohn spaces are sequential and sequential spaces are 
countably tight.

\end{document}